\documentclass[12pt]{extarticle}
\usepackage{amsmath, amsthm, amssymb, color}
\usepackage[colorlinks=true,linkcolor=blue,urlcolor=blue]{hyperref}
\usepackage{graphicx}
\usepackage{caption}
\usepackage{mathtools}
\usepackage{enumerate}
\usepackage{verbatim}
\usepackage{tikz,tikz-cd,tikz-3dplot}
\usepackage{amssymb}
\usetikzlibrary{matrix}
\usetikzlibrary{arrows}
\usepackage{algorithm}
\usepackage[noend]{algpseudocode}
\usepackage{caption}
\usepackage[normalem]{ulem}
\usepackage{subcaption}
\usepackage{multicol}
\oddsidemargin \evensidemargin
\usepackage{makecell}
\usepackage{array}
\usepackage{enumitem}

\newtheorem{theorem}{Theorem}[section]

\newtheorem{proposition}[theorem]{Proposition}

\theoremstyle{definition}

\newtheorem{examplex}[theorem]{Example}
\newenvironment{example}
{\pushQED{\qed}\examplex}
{\popQED\endexamplex}

\makeatletter
\def\Ddots{\mathinner{\mkern1mu\raise\p@
\vbox{\kern7\p@\hbox{.}}\mkern2mu
\raise4\p@\hbox{.}\mkern2mu\raise7\p@\hbox{.}\mkern1mu}}
\makeatother

\newcommand{\PP}{\mathbb{P}}
\newcommand{\RR}{\mathbb{R}}

\newcommand{\CC}{\mathbb{C} }

\title{\bf Likelihood Geometry of  \\ Determinantal Point Processes}

\author{Hannah Friedman, Bernd Sturmfels and Maksym Zubkov}
\date{}
\begin{document}
\maketitle

\begin{abstract}
  \noindent
 We study  determinantal point processes (DPP) through the lens of algebraic statistics.
We count the critical points of the log-likelihood function, and we compute them 
for small models, thereby disproving a
conjecture of Brunel, Moitra, Rigollet and Urschel.
\end{abstract}

\section{Introduction}

The determinantal point process (DPP) for discrete random variables
is a statistical model whose states are the subsets of a finite set
$[n] = \{1,2,\ldots,n\}$. This model is ubiquitous in probability theory, 
statistical physics, algebraic combinatorics, and machine learning \cite{Bor, Kul}.

This article
offers a study from the perspective of
algebraic statistics \cite{DSS}. 
Our motivation is the work on likelihood inference by
Brunel, Moitra, Rigollet and Urschel  \cite{BMRU}.
We shall answer their question  \cite[Conjecture 12]{BMRU} about
critical points of the log-likelihood function.

The DPP model is a semialgebraic set $\mathcal{M}_n$
of dimension $\binom{n+1}{2}$ in the simplex $\Delta_{2^n-1}$
whose points are the probability distributions on $2^{[n]}$.
The model $\mathcal{M}_n$ is parameterized by 
positive-definite symmetric $n \times n$ matrices $\Theta = (\theta_{ij})$.
In our model, the probability of observing a subset $I$
is proportional to the principal minor ${\rm det}(\Theta_I)$ indexed by that subset.
In symbols, 
\begin{equation}
\label{eq:DPPpara}
p_I \,\, = \,\, {\rm det}(\Theta_I)/ Z \qquad {\rm for} \,\, I \subseteq [n]. 
\end{equation}
Here $\det(\Theta_\emptyset) = 1$, and
the partition function is given by the sum of all $2^n$ principal minors:
\begin{equation}
\label{eq:ZZZ}
 Z \,\, = \,\,\sum_{I \subseteq [n]} \det(\Theta_I) \, \, = \,\, {\rm det}(\Theta + {\rm Id}_n),
 \end{equation}
where ${\rm Id}_n$ denotes the identity matrix of size $n \times n$.

We write $\mathcal{V}_n$ for the Zariski closure of $\mathcal{M}_n$ in
the complex projective space $\PP^{2^n-1}$. The model $\mathcal{M}_n$
and the variety $\mathcal{V}_n$ are cut out by the
{\em hyperdeterminantal ideal} that was studied by
Holtz-Sturmfels \cite{HS} and Oeding \cite{Oed}.
In the first non-trivial case $n=3$, our model is the zero set (in $\Delta_7$ or in $\PP^7$) of the
hyperdeterminant of format $2 \times 2 \times 2$, which is the quartic
\begin{equation}
\label{eq:hyperdet} \begin{matrix}
{\rm Det} \, = \,
  p_{000}^2p_{111}^2
+p_{001}^2 p_{110}^2
+p_{011}^2 p_{100}^2
+p_{010}^2 p_{101}^2
+ 4 p_{000} p_{011} p_{101} p_{110}
+4 p_{001} p_{010} p_{100} p_{111} &
\\
-\,2 p_{000} p_{001} p_{110} p_{111}
-2 p_{000} p_{010} p_{101} p_{111}
-2 p_{000} p_{011} p_{100} p_{111}  &
\\
-\,2 p_{001} p_{010} p_{101} p_{110}
-2 p_{001} p_{011} p_{100}    p_{110}
-2 p_{010} p_{011} p_{100} p_{101}. 
 \end{matrix}
\end{equation}
In (\ref{eq:hyperdet}), binary strings of length $3$ represent
subsets of $[3] = \{1,2,3\}$. For $n \geq 4$, one takes 
all occurrences \cite{Oed} of the hyperdeterminant (\ref{eq:hyperdet}) in
a tensor of format $2 {\times} 2 \times {\cdots} \times 2$ to cut out~$\mathcal{V}_n$.
A different representation of the variety $\mathcal{V}_n$ was found  by
Al Ahmadieh and Vinzant~\cite{AV}.

Data for the model $\mathcal{M}_n$ are given by a nonnegative
integer vector $u = (u_I: I \subseteq [n])$,
or equivalently, by a contingency table for $n$ binary states.
The log-likelihood function equals
\begin{equation}
\label{eq:loglikep}
L_u \,\, = \,\, \sum_{I \subseteq [n]} u_I \cdot {\rm log}(p_I) \,\,-\,\,
\sum_{I \subseteq [n]} u_I  \cdot {\rm log}\bigl(\sum_{I \subseteq [n]} p_I \bigr).
\end{equation}
In this formula, the $p_I$ are unknowns that serve as homogeneous coordinates
on the complex projective space $\PP^{2^n-1}$.
Note that $L_u$ is a multivalued function on $\PP^{2^n-1}$.
We consider~its restriction to the hyperdeterminantal variety $\mathcal{V}_n$.
The number of critical points of this restriction is  the
maximum likelihood degree (ML degree \cite{HS, HKT}). For small values of~$n$,
\begin{equation}
\label{eq:MLdegrees}
{\rm MLdegree}(\mathcal{V}_2) = 1 \quad {\rm and} \quad
{\rm MLdegree}(\mathcal{V}_3) = 13 .
\end{equation}
The number $13$ for the hyperdeterminant was computed in
\cite[Example 2.2.10]{DSS}. 
The number~$1$ arises because $\mathcal{V}_2 = \PP^3$.
But even this tiny case is interesting, as we shall see in Example \ref{ex:n=2}.

In  machine learning  \cite{BMRU, GJWX, Kul, SR},
one uses the parametric form of the log-likelihood:
\begin{equation}\label{eq:logliketheta}
  L_u \,\, = \,\, \bigl ( \sum_{I \subseteq [n]} u_I \cdot {\rm log}(\det(\Theta_I)) \bigr ) \, - \,
  \bigl(\sum_{I \subseteq [n]} u_I \bigr) \cdot \log(\det(\Theta + {\rm Id}_n)).
\end{equation}
Grigorescu, Juba, Wimmer and Xie \cite{GJWX}
showed that computing the maximum of this log-likelihood function is an NP-complete problem.
This had been conjectured by Kulesza \cite{Kul}.

The following example illustrates the
distinction between the formulations in
(\ref{eq:loglikep}) and  (\ref{eq:logliketheta}).

\begin{example}[$n=2$]  \label{ex:n=2} 
We write $u_\emptyset ,u_1,u_2,u_{12}$ for
the observed counts of subsets of~$\{1,2\}$.
The model $\mathcal{M}_2$ is given by the principal minors of a $2 \times 2 $ matrix
$\Theta =  [ \theta_{ij}] $, namely
$$ p_\emptyset = \frac{1}{Z},\,\,
 p_1 = \frac{\theta_{11}}{Z},\,\,
 p_2 = \frac{\theta_{22}}{Z} ,\,\,
 p_{12} = \frac{\theta_{11}\theta_{22}-\theta_{12}^2}{Z} , \quad
 {\rm where} \quad
  Z = \theta_{11}\theta_{22}-\theta_{12}^2+\theta_{11}+\theta_{22}+1.
$$
We view
$L_u = 
u_\emptyset {\rm log}(p_\emptyset) + u_1 {\rm log}(p_1) + u_2 {\rm log}(p_2) + 
u_{12} {\rm log}(p_{12})$
as a function of $\theta_{11},\theta_{12},\theta_{22}$.
Setting its partial derivatives to zero
gives three rational function equations in three unknowns.
The solutions $\hat \Theta$ of these equations are the critical points of $L_u$. 
A computation reveals
\begin{equation}
\label{eq:threesolutions} 
 \hat \theta_{11} = \frac{u_1}{u_\emptyset},\,
\hat \theta_{12} = \pm \frac{\sqrt{u_1 u_2{-}u_\emptyset u_{12}}}{u_\emptyset},\,
\hat \theta_{22} =  \frac{u_2}{u_\emptyset}
 \quad {\rm or} \quad
\hat \theta_{11} = \frac{u_1{+}u_{12}}{u_\emptyset {+} u_2},\,
\hat \theta_{12} = 0,\,
\hat \theta_{22} = \frac{u_2{+}u_{12}}{u_\emptyset{+}u_1}.
\end{equation}
From the implicit perspective, which was emphasized in \cite{HKT},
 there is only one solution. 
Its two preimages under the 2-1 parametrization of $\mathcal{M}_2$
are shown on the left in (\ref{eq:threesolutions}).
On the right in (\ref{eq:threesolutions}) is a ramification point of that 2-1 map.
It is not a critical point of (\ref{eq:loglikep}) on $\mathcal{V}_2 = \PP^3$.
\end{example}

This article is organized as follows.
In Section \ref{sec2}, we offer a detailed
investigation of the case $n=3$. 
In particular, we present our counterexample
to  \cite[Conjecture 12]{BMRU}.
The number of  critical points of
the parametric log-likelihood (\ref{eq:logliketheta}) is
found to be
$ 4 \cdot 13 + 2 \cdot 1 + 2 \cdot 1 + 2 \cdot 1 + 1 = 59$.
This is the $n=3$ analogue to the count
$2 \cdot 1 + 1 $ for the three solutions (\ref{eq:threesolutions})
 in Example~\ref{ex:n=2}.
 
 Section \ref{sec3} features
a general formula for counting complex critical points when
$n$ is arbitrary.
This involves contributions from all possible
block decompositions of the matrix $\Theta$.
These decompositions are called 
{\em partial decouplings} in the DPP literature.
Our Theorem \ref{thm:decoupling}
generalizes \cite[Theorem 11]{BMRU},
where it is assumed that the data vector $u$ lies on the model~$\mathcal{M}_n$.

In Section~\ref{sec4}, we apply numerical methods
to our problem.
Going well beyond~(\ref{eq:MLdegrees}),
we compute some solutions for $n=4,5$
with the software {\tt HomotopyContinuation.jl} \cite{BT}.
After replacing (\ref{eq:DPPpara}) with a birational parametrization, we
apply the monodromy method for rational likelihood equations
in \cite{ABF, ST}.
The focus is on solutions that are real and
positive-definite.

While this paper focuses primarily on the
algebraic structure and
likelihood geometry
of the DPP model,
it does have the potential to contribute to applications in 
statistics \cite{SR} and machine learning \cite{BMRU, Kul}.
The MLE problem we solve differs from the 
subset selection problem, and it can viewed
as learning the best parameters for 
a given sample of subsets.

\section{Three-by-three Matrices }
\label{sec2}

In this section, we examine the likelihood geometry of the DPP model with $n=3$.
The model parameters are the six entries $\theta_{ij}$ of the symmetric
$3 \times 3$ matrix $\Theta$. 
Fix any data vector
$\, u =  (u_\emptyset, u_1,u_2,u_3, u_{12}, u_{13}, u_{23}, u_{123} )$.
The sum of its eight coordinates~$u_I$ is the sample size of the data,
here denoted $|u|$. The
parametric log-likelihood function (\ref{eq:logliketheta}) equals
\begin{equation}
\label{eq:loglike3}
 \begin{matrix}
\!\! L_u & \!\! = \!\! & u_1\, {\rm log}( \theta_{11}) + u_2 \, {\rm log}(\theta_{22}) 
+ u_3 \, {\rm log}(\theta_{33})  
+ \,u_{12} \, {\rm log}( \theta_{11} \theta_{22} {-} \theta_{12}^2) 
+ u_{13}\, {\rm log}( \theta_{11} \theta_{33} {-} \theta_{13}^2)  \\ & & 
+ \,u_{23} \, {\rm log}( \theta_{22} \theta_{33} {-} \theta_{23}^2) 
\,+\, u_{123}\, {\rm log}({\rm det}(\Theta)) \,\,
-\,\, |u| \, {\rm log} \bigl({\rm det}(\Theta + {\rm Id}_3)\bigr).
\end{matrix}
\end{equation}
Setting the partial derivatives of $L_u$ to zero
gives six rational function equations in six unknowns.
The solutions to these equations are the critical points of $L_u$. 
Our theory in Section \ref{sec4} predicts $59$
critical points when the data vector $u$ is generic.
Among these are $4 \cdot 13 = 52$ solutions arising from
critical points of (\ref{eq:loglikep}) on 
the hyperdeterminantal hypersurface $\mathcal{V}_3$.
These critical points $\hat \Theta$ have
$\hat \theta_{ij} \not= 0$ for all $i,j$.
The clusters of four arise by multiplying
 two of the three off-diagonal entries
$\hat \theta_{ij}$ by $-1$. This does not change the
distribution $\hat p \in \mathcal{M}_3$.

The other $7 = 1 + 3(2 \cdot 1) $
critical points of (\ref{eq:logliketheta})
are extraneous, in the sense that they do not come
from  critical points of (\ref{eq:loglikep}) on $\mathcal{V}_3$.
This is analogous to the point on 
the right in~(\ref{eq:threesolutions}).
One of the
seven special solutions is the diagonal matrix
$\, \hat \Theta = {\rm diag} (\hat \theta_{11}, \hat \theta_{22}, \hat \theta_{33})$, where
\begin{equation}
\label{eq:diagonals}
\hat \theta_{11} = \frac{u_1 \! +\! u_{12} \! +\!  u_{13} \! +\! u_{123}}{
u_\emptyset + u_{2} + u_{3} + u_{23}},\,\,\,
\hat \theta_{22} = \frac{u_2 \! +\! u_{12} \! +\!  u_{23} \! +\! u_{123}}{
u_\emptyset + u_{1} + u_{3} + u_{13}}, \,\,\,
\hat \theta_{33} = \frac{u_3 \! +\! u_{13} \! +\!  u_{23} \! +\! u_{123}}{
u_\emptyset + u_{1} + u_{2} + u_{12}}. 
\end{equation}
The other six special solutions come in three pairs, one for each 
decomposition of a $3 \times 3$ matrix into a $2 \times 2 $ block
and a $1\times 1$ block. One such pair of critical points has the form
\begin{equation}
\label{eq:block21a} \hat \Theta = 
\begin{small} \begin{bmatrix}
\hat \theta_{11} & \hat \theta_{12} & 0  \\
\hat \theta_{12} & \hat \theta_{22} & 0 \\
0 & 0 & \hat \theta_{33} \end{bmatrix}, \end{small}
\end{equation}
where $\hat \theta_{33}$ is the expression on the right in (\ref{eq:diagonals}), 
and the other three entries in (\ref{eq:block21a}) are 
\begin{equation}
\label{eq:block21b}
\hat \theta_{11} =  \frac{u_1+u_{13}}{u_\emptyset + u_3},\,\,
\hat \theta_{12} = \, \pm \frac{\sqrt{(u_1{+}u_{13})(u_2{+}u_{23})- (u_\emptyset{+}u_3)
(u_{12}{+}u_{123})} }{u_\emptyset + u_3},\,\,
\hat \theta_{22} = \frac{u_2+u_{23}}{u_\emptyset + u_3}.
\end{equation}

We now move away from the hypothesis that the data are generic.
Namely, we make the assumption that $u$ lies on the variety $\mathcal{V}_n$.
In fact, we even assume that $\frac{1}{u_\emptyset} u $  lies in the
model, meaning that it
is the vector of principal minors of
some real symmetric $n \times n$ matrix.
This is the standing assumption on the data in the article \cite{BMRU}
to which we shall turn shortly.

\begin{example}[Data in the model] \label{ex:datainmodel}
We consider the following data for $n=3$:
\begin{equation}
\label{eq:onmodeldata}
    (u_\emptyset, u_{1}, u_2, u_3, u_{12}, u_{13}, u_{23}, u_{123}) \,\,= \,\,
    (1, 8, 22, 18, 151, 135, 360, 2412)
\end{equation}
This vector lies in the model because its entries are the principal minors of
any of the matrices
\begin{equation}
\label{eq:tautological} \begin{small}
    \begin{bmatrix}
    \,  8 & 5  & 3\\
    \,  5 & 22 & 6\\
    \,  3 & 6  & 18 \, 
    \end{bmatrix},\,\,
        \begin{bmatrix}
      \phantom{-} 8 \! &\! -5  \!&\! -3\\
      -5 \!&\! \phantom{-} 22 \! & \!\phantom{-} 6\\
      -3 \! &\! \phantom{-}6  \! &\! \phantom{-}18 \, 
    \end{bmatrix},\,\,
    \begin{bmatrix}
    \, \phantom{-} 8 \! & \! -5  \! & \!  \phantom{-} 3\\
    \,  -5 \! & \!  \phantom{-} 22 \! & \! -6\\
    \,   \phantom{-} 3 \! & \! -6  \! & \!  \phantom{-} 18 \, 
    \end{bmatrix},\,\,
    \begin{bmatrix}
    \,   \phantom{-} 8 \! & \!  \phantom{-}  5  \! & \! -3\\
    \,   \phantom{-} 5 \! & \!  \phantom{-} 22 \! & \! -6\\
    \,  -3 \! & \! -6  \! & \!  \phantom{-} 18 \, 
    \end{bmatrix}. \end{small}
\end{equation}
   By construction, these are the  four global maxima of the log-likelihood function 
in (\ref{eq:loglike3}), and they map to the global maximum of (\ref{eq:loglikep})
on $\mathcal{M}_3$.
Among the other $12$ complex critical points on $\mathcal{V}_3$,
six are real and lie on $\mathcal{M}_3$. Four of these correspond to the positive-definite matrices
  \begin{align*}
    \begin{footnotesize}
      \begin{matrix} 
        \hat \theta_{11} =  7.72799090116006 & \hat \theta_{12} =  4.14366972540362\\
        \hat \theta_{13} =  1.87300176302618 & \hat \theta_{22} =  20.1464857136673\\
        \hat \theta_{23} =  0.82526924316919 & \hat \theta_{33} =  16.4735825997691
      \end{matrix}
    \end{footnotesize}
    \quad && \quad
    \begin{footnotesize}
      \begin{matrix} 
        \hat \theta_{11} = 6.92478592243203 & \hat \theta_{12}=  0.42796700405714\\
        \hat \theta_{13} = 1.80374923458180 & \hat \theta_{22} =  19.2531487326101\\
        \hat \theta_{23} = 4.44778298807768 & \hat \theta_{33} =  17.4175047796638
      \end{matrix}\,
    \end{footnotesize}
    \smallskip \\
    \begin{footnotesize}
      \begin{matrix} 
        \hat \theta_{11} = 7.56880693316022  & \hat \theta_{12} = 4.28496510066628 \\
        \hat \theta_{13} = 0.86306207237349 &  \hat  \theta_{22} = 21.3776618810445 \\
        \hat \theta_{23} = 4.79253523095731 & \hat  \theta_{33} = 17.0982120638953
      \end{matrix}
    \end{footnotesize}
    \quad && \quad
    \begin{footnotesize}
      \begin{matrix} 
        \hat \theta_{11} = 7.57820456385679 & \hat \theta_{12} = -3.8397212783772\\
        \hat \theta_{13} = 0.98046698151082 & \hat \theta_{22} = 20.9281938578911\\
        \hat \theta_{23} = 3.86656494286390 & \hat \theta_{33} = 17.1007249363163.
      \end{matrix}
    \end{footnotesize}
      \end{align*}
In addition to these,
 the parametric log-likelihood (\ref{eq:loglike3}) has seven more critical points
with a block structure.  These are obtained by substituting
(\ref{eq:onmodeldata}) into the formulas
(\ref{eq:diagonals}), (\ref{eq:block21a}), (\ref{eq:block21b}).
\end{example}

We now turn to the approach of
Brunel, Moitra, Rigollet and Urschel \cite{BMRU},
and we present a counterexample to 
\cite[Conjecture 12]{BMRU}, which states that
there are no critical points other than those
obtained from {\em partial decouplings}.
Partial decouplings correspond to the
block decompositions we saw in (\ref{eq:threesolutions}),
(\ref{eq:diagonals}) and (\ref{eq:block21a}).
We discuss these in Section \ref{sec3} for general $n$.

The set-up in \cite{BMRU}
uses the parametric form (\ref{eq:logliketheta}) of the log-likelihood,
i.e.~$L_u$ is a function on the cone
of positive-definite $n \times n$ matrices.
Furthermore, \cite{BMRU} assumes
that the data vector $u$ is sampled from the model $\mathcal{M}_n$.
When $u$ is given by the principal minors of
some  positive-definite $n \times n$ matrix,
\cite[Theorem 11]{BMRU} shows that all partial decouplings
are critical points of $L_u$. 
This includes the
empirical distribution $\hat p = \frac{1}{|u|} u $,
which is the only critical point from partial decouplings
with full support.
There are exponentially many other such
critical points, and the question is whether these are all.
We show that the answer is negative.

\begin{proposition} 
For $n=3$, the log-likelihood function in $\Theta$  given by some $u \in \mathcal{M}_3$
has critical points that do not correspond to partial decouplings.
This resolves \cite[Conjecture 12]{BMRU}.
\end{proposition}

\begin{proof}
The proof is furnished by Example \ref{ex:datainmodel}.
The log-likelihood function  $L_u$
for  $u$ in (\ref{eq:onmodeldata})
 has $28$  fully-supported real critical points, $20$ of which are positive-definite.
 The four matrices in (\ref{eq:tautological})
 are the global maxima.
 Below that, we show matrix representatives for four positive-definite critical points
  not corresponding to partial decouplings.
\end{proof}

\section{Partial Decouplings}
\label{sec3}

We saw that some of the critical points
of the parametric log-likelihood (\ref{eq:logliketheta})
are matrices $\hat \Theta$ with a block decomposition.
These were called partial decouplings in \cite{BMRU}.
Such critical points were characterized in \cite[Theorem 11]{BMRU}, 
under the hypothesis that
the data vector $u$ lies in the model $\mathcal{M}_n$.
In what follows we offer a generalization of that result.
We no longer assume $u \in \mathcal{M}_n$. From now on,
we allow  $u = (u_I:  I \subseteq [n])$ to be any complex
vector of length $2^n$.
If $ u$ is generic, then all critical points of (\ref{eq:loglikep})
on $\mathcal{V}_n$ have fully supported preimages under the principal minor map.
We write $\mu_n$ for the ML degree of the projective variety $\mathcal{V}_n$.
We know from (\ref{eq:MLdegrees}) that
 $\mu_1 = \mu_2 = 1$ and $ \mu_3 = 13$.
 In the next section we shall show that
     $\mu_4 = 3526$.
  
 Recall that a {\em set partition} of $[n] = \{1,2,\ldots,n\}$ is a set
$\pi = \{\pi_1,\ldots,\pi_k\}$, where
the $\pi_i$ are non-empty pairwise disjoint
subsets of $[n]$ whose union equals $[n]$.
We write $\mathcal{P}_n$ for the set of all set partitions of $[n]$.
The cardinality $|\mathcal{P}_n|$ is the {\em Bell number}, which equals
$2, 5, 15, 52, 203, 877, \ldots$ for $n=2,3,4,5,6,7$. See Examples \ref{ex:part3}
and \ref{ex:part4} for the cases $n=3,4$.

\begin{theorem} \label{thm:decoupling}
The critical points $\hat \Theta$  of the parametric log-likelihood function $L_u$ 
in (\ref{eq:logliketheta}) are
found by solving various likelihood equations on submodels $\mathcal{M}_r$ for $r \leq n$. 
If $u$ is generic, in the sense of algebraic geometry, then
the total number of complex critical points of $L_u$ equals
\begin{equation}
\label{eq:partitionsum} \sum_{\pi \in \mathcal{P}_n}
\prod_{i=1}^{|\pi|} \bigl( \,2^{|\pi_i|-1} \,\mu_{|\pi_i|} \,\bigr).  \end{equation}
\end{theorem}

Given a partition $\pi = \{\pi_1,\ldots,\pi_k\}$, the upper limit $|\pi| = k$ is the number of parts of $\pi$. 
The phrase ``generic in the sense of algebraic geometry''
means that there exists a proper subvariety in the data space $\RR^{2^n}$
such that the statement holds for all vectors $u$ outside that variety.
In particular, it holds with probability one for a randomly selected vector $u \in \RR^{2^n}$.

\begin{example}\label{ex:part3}
For $n=3$, we have
$ \mathcal{P}_3 = \bigl\{ 
\{1,2,3\}, \{12,3\}, \{13,2\}, \{23,1\}, \{123\} \bigr\}$.
Hence the number  (\ref{eq:partitionsum}) of critical points equals
$ 1 \cdot 1 \, + \, 2 \cdot 1 \, + \, 2 \cdot 1 \,+ \, 2 \cdot 1 \, + \,4 \cdot 13 \,\, = \,\, 59$.
These $59$ solutions were described in Section \ref{sec2}, with an
explicit numerical instance in Example \ref{ex:datainmodel}.
\end{example}

\begin{example}\label{ex:part4}
For $n=4$, there are $28,441$ critical points. The sum (\ref{eq:partitionsum}) 
is over the $15$ set partitions of $[4]$. The biggest summand is
 $\,2^{4-1} \cdot 3526 = 28,208$, for $\pi = \{1234\}$
 with $k=1$. The partitions with $k \geq 2$ contribute the summands
$52,52,52,52,4,4,4,2,2,2,2,2,2,1$.
\end{example}

\begin{proof}[Proof of Theorem \ref{thm:decoupling}]
We fix one partition $\pi = \{\pi_1,\ldots,\pi_k\}$ in $\mathcal{P}_n$.
Suppose that $\Theta$ is a symmetric $n \times n$ matrix
that has the block structure $\pi$, so some of the entries are zero.
However, all nonzero entries of $\Theta$ are distinct unknowns.
We write $\Theta = \Theta_{\pi_1} \oplus \Theta_{\pi_2} \oplus\, \cdots \,\oplus \Theta_{\pi_k}$.

Let $L_u(\Theta)$ denote the evaluation of the
log-likelihood function at the block matrix $\Theta$. Then,
$L_u(\Theta)$ is a function in $\sum_{i=1}^k \binom{|\pi_i|+1}{2}$ unknowns,
namely the entries of the $k$ blocks $\Theta_{\pi_i}$.
Assuming $u$ to be generic, we count critical points
$\hat \Theta$ for which all entries in the blocks $\hat \Theta_{\pi_i}$ are nonzero.
We claim that the total number of these critical points is equal to
\begin{equation}
\label{eq:proposedcount}
    \prod_{i = 1}^k \bigl (2^{|\pi_i| - 1} \mu_{|\pi_i|} \bigr ).
\end{equation}

Our $\pi$-restricted log-likelihood function admits an additive decomposition
\begin{equation}
\label{eq:additivedecomp}
    L_u(\Theta) \,\,=\,\, \sum_{i = 1}^k L_{v^{(i)}}(\Theta_{\pi_i}).
\end{equation}
  Here $v^{(i)}$ is a vector in $\RR^{2^{|\pi_i|}}$, indexed by subsets of $\pi_i$,
  that is obtained from $u$ by a linear transformation.
    To see this, we use the following identity for the minors of our block matrix:
  \begin{align*}
    \log(\det(\Theta_I)) \,\,=\,\, \log\bigl( \,\prod_{i = 1}^k \det(\Theta_{I \cap \pi_i}) \, \bigr)
    \,\, = \,\,\sum_{i = 1}^k \log(\det(\Theta_{I \cap \pi_i}))
    \qquad \hbox{for all} \,\, I \subseteq [n]. 
  \end{align*}
  The analogous decomposition holds for the log-partition function
$\,{\rm log}(Z)=  \log(\det(\Theta + {\rm Id}_n))$.
  From this we conclude that (\ref{eq:additivedecomp}) holds
  if we define the restricted data vector $v^{(i)}$ as follows:
\begin{equation}
\label{eq:definev}
      v^{(i)}_J  \,\, = \,\, \sum \bigl\{ u_I : I \subseteq [n] 
  \,\,{\rm and} \,\, I \cap \pi_i = J \bigr\} 
\quad  \hbox{ for all $J \subseteq \pi_i$}. 
\end{equation}

The number of fully supported critical points of
$L_{v^{(i)} }(\Theta_{\pi_i})$ is equal to
$2^{|\pi_i|-1} \mu_{|\pi_i|}$. Indeed, the data
vector $v^{(i)}$ is still generic, and we are computing
critical points on the variety $\mathcal{V}_{|\pi_i|}$.
Each critical point on  $\mathcal{V}_{|\pi_i|}$
comes from a cluster of $2^{|\pi_i|-1}$
critical matrices $\hat \Theta_{\pi_i}$.
Since the summands in (\ref{eq:additivedecomp})
involve disjoint sets of unknowns, these critical points
combine for $i=1,\ldots,k$. Therefore, the total number of
critical points of $L_u(\Theta)$ is the product in (\ref{eq:proposedcount}).

The next step  is to show
that the points above are critical points of $L_u(\Theta)$,
where $\Theta$ is now an $n {\times} n$ matrix with all $\binom{n+1}{2}$
entries distinct unknowns. 
To see this, consider the partial derivative of $L_u(\Theta)$
with respect to any off-diagonal parameter $\theta_{ij}$.
This partial derivative is
an element of the ring $R$ 
that is obtained by localizing
 the polynomial ring $\RR[\Theta]$ at the
product of $Z$ and all principal minors of $\Theta$.
We claim that that
$\partial L_u/ \partial \theta_{ij}$
lies in the following 
ideal of $R$, where the intersection is over all 
$2^{n-2}$ partitions $K \cup L = [n]$
with $i \in K$ and~$j \in L$:
\begin{equation}
\label{eq:idealintersection}
\partial L_u/ \partial \theta_{ij} \,\,\,\in 
\bigcap_{K \ni i, L \ni j} \, \bigl\langle\, \theta_{kl} \,: \,k \in K, l \in L \bigr\rangle.
\end{equation}

Let $\mathcal{I}$ be the ideal generated by the monomials in $\partial {\rm det}(\Theta)/\partial \theta_{ij}$.
We claim that the ideals on the right side of
(\ref{eq:idealintersection}) are associated primes of $\mathcal{I}$.
  Namely, $\langle\, \theta_{kl} \,: \,k \in K, l \in L \rangle$  $= (\mathcal I  : a)$ where $a = \prod \{\theta_{h_1, h_2} \, : \, (h_1,h_2) \in   K {\times} K \cup L {\times} L\}$.  
Since $\mathcal I$ is a monomial ideal, it suffices to show the
following: if $m$ is a monomial, then there exists a monomial in $\partial {\rm det}(\Theta)/\partial \theta_{ij}$ which divides $ma$  if and only if $\theta_{kl}$ divides $m$ from some $k \in K$ and $l \in L$.
Every monomial in the determinant has the form $ s= \prod_{h=1}^n \theta_{h\sigma(h)}$ for some permutation $\sigma$.
  Suppose that $\partial s/ \partial \theta_{ij}$ is nonzero and divides $ma$.
  Then $\sigma(i)$ must be $j$.
  Since $i \in K$ and $j \in L$, for $\sigma$ to be a bijection, there must be some $k \in K$ and $l \in L$ such that $\sigma(l) = k$.
  Since $\theta_{l k}$ cannot divide $a$, but $\theta_{l k}$ divides $\partial s/ \partial \theta_{ij}$, which, in turn, divides $ma$, it follows that $\theta_{l k} = \theta_{kl}$ divides $m$. 

  Now suppose $\theta_{kl}$ divides $m$ and let $s = \prod_{h=1}^n \theta_{h\sigma(h)}$ where $\sigma = (i\ j\ l\ k)$ if $k \neq i$ and $l \neq j$, $\sigma = (i\ j\ l)$ if $k = i$ and $l \neq j$, $\sigma = (i\ j\ k)$ if $k \neq i$ and $l = j$, and $\sigma = (i\ j)$ if $k = i$ and $l = j$. 
  Up to scaling,  $\partial s/\partial \theta_{ij} = s/\theta_{ij}$, which divides $ma$, as $\theta_{kl}$ divides $m$ and $s/(\theta_{ij}\theta_{k l})$ divides $a$.
  This concludes the proof of our claim that $\,\langle\, \theta_{kl} \,: \,k \in K, l \in L \rangle$  $= (\mathcal I  : a)$.
  
The same statement holds when $\Theta$ is replaced by $\Theta+{\rm Id}_n$
or any principal submatrix $\Theta_I$.
Since $\partial L_u/ \partial \theta_{ij}$ is in the ideal of these determinants,
we have established the inclusion~(\ref{eq:idealintersection}).

Now, fix any set partition $\pi \in \mathcal{P}_n$,
and suppose that $i$ and $j$ lie in distinct blocks of $\pi$.
The partial derivative $\partial L_u/ \partial \theta_{ij} $
vanishes identically when the full matrix $\Theta$ is replaced by the block matrix
$\,\Theta_{\pi_1} \oplus\, \cdots \,\oplus \Theta_{\pi_k}$.
This follows from (\ref{eq:idealintersection}).
This vanishing property shows that the block matrices
$\,\hat \Theta_{\pi_1} \oplus\, \cdots \,\oplus \hat \Theta_{\pi_k}$
derived above are, in fact, critical points of $L_u(\Theta)$.

At this point, we know that (\ref{eq:proposedcount})
is a lower bound for the number of critical points.
The final step in our proof is to show that
no further critical points exist. To see this,
let $\hat \Theta$ be any critical point of
the parametric log-likelihood (\ref{eq:logliketheta}).
Suppose that the support of $\hat \Theta$ is not
contained in any proper block structure.
Then its fiber over $\mathcal{V}_n$
consists of $2^{n-1}$ distinct matrices, 
which are reduced points in that fiber.
This implies that the common  image in $\mathcal{V}_n$ of the
$2^{n-1}$ matrices 
is a critical point of (\ref{eq:loglikep}).
Hence $\hat \Theta$ has
been counted  in (\ref{eq:proposedcount}),
by the summand for $\pi = \{[n]\}$.
If $u$ is generic then we can 
conclude that $\hat \Theta$ has no zero coordinates.

It remains to consider critical points $\hat \Theta$ 
that conform to the block structure for some partition  $K \cup L = [n]$,
i.e.~$\hat \theta_{kl} = 0$ for all $k \in K$ and $l \in L$.
We now apply the previous argument inductively
to the respective blocks $\hat \Theta_K$ and $\hat \Theta_L$,
and eventually we arrive at
$\hat \Theta =  \hat \Theta_{\pi_1} \oplus \cdots \oplus \hat \Theta_{\pi_k}$
for some partition $\pi$ of $[n]$.
This means that $\hat \Theta$ was counted in (\ref{eq:proposedcount}).
\end{proof}

We have shown that, for generic data vectors $u$,
the support of each critical point $\hat \Theta$
is precisely given by one of the block structures.
This property can fail when $u$ is not generic.

\begin{example} \label{ex:hannah}
Fix $n=3$ and $u = (2, 1, 3, 7, 9, 10, 19, 22)$.
Then $L_u$ has $59$ distinct critical points,
as in Example \ref{ex:part3},
with $52$  from the trivial partition $\pi = \{123\}$.
One of these~is
$$ \hat \Theta \,=\, \begin{small}
\begin{bmatrix}
2 & 0 & 2 \\
0 & 4 & 3 \\
2 & 3 & 7
\end{bmatrix} .\end{small}
$$
The zero entry $\hat \theta_{12}= 0$ is accidental, not due to
any block structure.  Here, $u$ is not generic.
\end{example}

It is now instructive to revisit the 
implicit formulation of our MLE problem.
We seek points $p$ on the hyperdeterminant 
$V({\rm Det}) \subset \PP^7 $ such that the following matrix has rank~$\leq 2$:
\begin{equation}
\label{eq:3by8}
\begin{bmatrix}
u_\emptyset & u_1 & u_2 & u_3 & u_{12} & u_{13} & u_{23} & u_{123}
 \\
p_{000} & p_{100} & p_{010} & p_{001} & p_{110} & p_{101} & p_{011} & p_{111} \smallskip \\
p_{000} \frac{\partial{{\rm Det}}}{\partial{p_{000}}} &
p_{100} \frac{\partial{{\rm Det}}}{\partial{p_{100}}} &
p_{010} \frac{\partial{{\rm Det}}}{\partial{p_{010}}} &
p_{001} \frac{\partial{{\rm Det}}}{\partial{p_{001}}} &
p_{110} \frac{\partial{{\rm Det}}}{\partial{p_{110}}} &
p_{101} \frac{\partial{{\rm Det}}}{\partial{p_{101}}} &
p_{011} \frac{\partial{{\rm Det}}}{\partial{p_{011}}} &
p_{111} \frac{\partial{{\rm Det}}}{\partial{p_{111}}} 
\end{bmatrix}.
\end{equation}
We require each coordinate
 $p_{ijk}$ to be  non-zero, and also
$\sum_{ijk} p_{ijk} \not= 0$.
We further disallow $p$ to lie in the singular locus of 
$V({\rm Det})$, i.e.~the three  flattenings of
the $2 \times 2 \times 2$ tensor $p$ are $2 \times 4$ matrices of rank $2$.
This system has $13$ solutions, even for
the special $u$ in Example~\ref{ex:hannah}.

\section{Numerical Computations}
\label{sec4}

We now discuss the solution of the likelihood equations
using methods from numerical algebraic geometry.
For our computations we use the software
{\tt HomotopyContinuation.jl} due to
Breiding and Timme \cite{BT}, along
with the certification feature  in \cite{BRT}.
Our approach is based on the
 monodromy method for rational likelihood equations
 that was developed in \cite{ABF, ST}.

The underlying idea is as follows. We consider the likelihood equations
$\nabla L_u(\Theta) = 0$ where both $u$ and $\Theta$ are unknowns.
These define the {\em likelihood correspondence} \cite[Definition 1.5]{HS}.
In our situation, the likelihood correspondence has many irreducible
components, one for each set partition $\pi \in \mathcal{P}_n$.
This is the geometric interpretation of Theorem \ref{thm:decoupling}.
We wish to focus on the main component, for $\pi = \{[n]\}$,
which comprises the critical points of (\ref{eq:loglikep}) restricted to $\mathcal{V}_n$.
Luckily, numerical algebraic geometry does this for us automatically.

The likelihood equations are linear in $u$.
We thus can fix a random complex matrix $\Theta$, 
and then solve for a matching $u$. Afterwards, we fix $u$
and we vary $\Theta$. By running monodromy loops
in {\tt HomotopyContinuation.jl},
one eventually finds all solutions $\hat \Theta$ to 
$\nabla L_u(\Theta) = 0$  for that fixed $u$.
Here ``all'' means all critical points of  (\ref{eq:loglikep}) on $\mathcal{V}_n$,
because the monodromy loops stay on the main
irreducible component of the likelihood correspondence.

The program terminates after a  heuristic criterion
is satisfied. If this happens, then we can be confident that
all solutions have been found, and that the number of
solutions is equal to $\mu_n = {\rm MLdegree}(\mathcal{V}_n)$.
However,  there is still a tiny
chance that some solutions have been missed, which would mean that
the true $\mu_n$ is a little larger than the current count.
At this stage, we apply the command {\tt certify}
which generates a proof, based on interval arithmetic,
that all floating-point approximations that were found are, in fact,
distinct solutions \cite{BRT}. 

The pipeline described above proves that the number we found
is a lower bound for $\mu_n$. To prove that it is also an upper bound,
one would need some insights from intersection theory.
But this is still missing for many statistical models, including the one treated in this paper.
The process described above is quite fast for $n=4$, and it yields the following result.

\begin{proposition}
The ML degree of the DPP model $\,\mathcal{M}_4$ satisfies $\mu_4 \geq 3526$.
Based on our numerical computation, we are confident that $\mu_4 = 3526$.
\end{proposition}

The principal minor map is $2^{n-1}$-to-$1$.
For our computations 
we use a  reparametrization which makes the map $1$-to-$1$, 
reducing the number of paths to be
tracked in {\tt HomotopyContinuation.jl}
 by a factor of $2^{n-1}$.
We first show the new coordinates for~$n=3$.

\begin{example}[Birational Reparametrization]\label{ex:reparametrization}
  We reparametrize our matrix 
\begin{equation}
\label{eq:threebythree}
  \Theta \,\,=\,\,
\begin{small} 
\begin{bmatrix}
      \theta_{11} & \theta_{12} & \theta_{13}\\
      \theta_{12} & \theta_{22} & \theta_{23}\\
      \theta_{13} & \theta_{23} & \theta_{33}      
    \end{bmatrix}
    \end{small}
\end{equation}
  so that the principal minor map 
$
\Theta     \mapsto 
    \bigl(  \theta_{11}, \theta_{22}, \theta_{33}, \theta_{11}\theta_{22} {-} \theta_{12}^2, \theta_{11}\theta_{33} {-} \theta_{13}^2, \theta_{22}\theta_{33} {-} \theta_{23}^2, 
    {\rm det}(\Theta) \bigr)
$
becomes injective by  replacing the monomials
  $\theta_{12}^2, \theta_{13}^2, \theta_{23}^2, \theta_{12}\theta_{13}\theta_{23}$ with new variables.
    These four monomials are algebraically dependent, so we introduce three new variables:
$\   x_{12} = \theta_{12}^2$,
$    x_{13} = \theta_{13}^2$, and
$    x_{23} = \theta_{12}\theta_{13}\theta_{23}.$
Solving for the $\theta_{ij}$, we now
substitute the following into~(\ref{eq:threebythree}):
  \begin{align*}
    \theta_{12} \,=\,  \sqrt{x_{12}} && 
    \theta_{13} \,=\, \sqrt{x_{13}} &&
    \theta_{23} \,=\, x_{23}/\sqrt{x_{12}x_{13}}.
  \end{align*}
  This yields a birational map between $\CC^6$ 
  and the hypersurface $V({\rm Det})$ in $\PP^7$.
  \end{example}

The general case is similar. For $n \geq 4$, we
 replace all off-diagonal parameters as follows:
$$
\hbox{For $i \not= j$ we set} \quad
  \theta_{ij} \,= \,
  \begin{cases}\,
    \sqrt{x_{ij}} &\quad \text{if $i = 1$},\\ 
\,    x_{ij}/\sqrt{x_{1i}x_{1j}} &\quad \text{otherwise.}
  \end{cases}
$$
After this, the log-likelihood  $L_u$ is
a function in the $n$ diagonal entries $\theta_{ii}$
and the $\binom{n}{2}$ new variables $x_{ij}$.
Its partial derivatives give a system of
$\binom{n+1}{2}$  rational function equations in $\binom{n+1}{2}$ variables.
We now use the command \verb+monodromy_solve+ to solve this system.
We find and certify $13$ complex solutions for $n = 3$ and $3526$ complex solutions for $n = 4$.
The $n = 3$ computations run in under a second. For $n = 4$, the computation takes about $20$~minutes. 
These times can be improved significantly by using multiple threads in {\tt Julia}.

In the statistical application to DPP, we seek critical points that are real, not complex.
Ideally, we want $\hat \Theta$ to be positive-definite.
We ran the above computation on many data vectors $u$,
with the aim of maximizing the number of real critical points.
Here is one winner:

\begin{example}[11 Positive-Definite Critical Points]\label{ex:11real}
Fix the data
   $u= (1, 5, 5, 5, 5, 5, 5, 1)$.
The likelihood function $L_u$ has two complex critical points.
The remaining $11$ critical points correspond to positive-definite matrices. 
%%  The $11$ real critical points $\hat \Theta$ are all positive-definite.
  Nine come from by permuting indices on three $\hat \Theta$:
  \begin{align*} \!\!
    \begin{footnotesize}
      \begin{array}{ll} 
        \hat \theta_{11} = 6.0   \! & \! \hat \theta_{12} = 4.4721360\\
        \hat \theta_{13} = 4.472136 \! &\! \hat \theta_{22} = 3.8888889\\
        \hat \theta_{23} = 3.111111 \! & \! \hat \theta_{33} = 3.8888889
      \end{array}
    \end{footnotesize}
      && 
    \begin{footnotesize}
      \begin{array}{ll}
        \hat \theta_{11} = 2.142857\! & \! \hat \theta_{12} = 0.8571429 \\
        \hat \theta_{13} = 2.236068 \! & \! \hat \theta_{22} = 2.1428571\\
        \hat \theta_{23} = 2.236068 \! & \! \hat \theta_{33} = 3.5
      \end{array}
    \end{footnotesize}
 & & 
    \begin{footnotesize}
      \begin{array}{ll}
        \hat \theta_{11} = 5.0  \! &\! \hat \theta_{12} = 3.4641016 \\
        \hat \theta_{13} = 3.4641016  \! &\! \hat \theta_{22} = 3.0 \\
        \hat \theta_{23} = 2.0  \! &\! \hat \theta_{33} = 3.0.
      \end{array}
    \end{footnotesize}
  \end{align*}
The remaining two critical points $\hat p \in \mathcal{M}_3$ are invariant under permuting the indices $1,2,3$:
\begin{align*}
  \begin{footnotesize}
      \begin{array}{ll}
        \hat \theta_{11} = 5.652906131 & \hat \theta_{12} = 5.265758657 \\
        \hat \theta_{13} = 5.265758657  & \hat \theta_{22} = 5.652906131\\
        \hat \theta_{23} = 5.265758657  & \hat \theta_{33} = 5.652906131
      \end{array}
  \end{footnotesize}
  && \quad &&
    \begin{footnotesize}
      \begin{array}{ll}
        \hat \theta_{11} = 1.742592619   & \hat \theta_{12} = -0.840402407\\
        \hat \theta_{13} = -0.840402407  & \hat \theta_{22} = 1.742592619\\
        \hat \theta_{23} = -0.840402407 & \hat \theta_{33} = 1.742592619.
      \end{array}
    \end{footnotesize}
\end{align*}
Among the $11$ critical points, five are local maxima, two of which are global maxima.
The global maxima come from the matrices that are invariant under permuting indices, i.e.~the matrices with constant diagonal and off-diagonal entries.
The log-likelihood evaluates to $-63.46051485$ at these two points.
The other local maxima come from permutations of~the matrix with $\hat\theta_{11} = 5.0$.
The log-likelihood evaluates to $-63.63109767$ at these three points. 
\end{example}

We now turn to the case $n=4$, where 
numerical accuracy is already a notable challenge.

\begin{example}[$n=4$]
  We fix
$u \,=\, (u_\emptyset, u_1,\ldots,u_4, u_{12}, \ldots,u_{34},
u_{123}, \ldots, u_{234}, u_{1234})$ to be
$$ u \,\,=\,\,
(1,12,12,12,12,12,12,12,
12,12,12,12,12,12,12,1).
$$
We certified $3221$ complex critical points $\hat p$ for the function $L_u$ on the $10$-dimensional
variety $\mathcal{V}_4$. Among these $3221$, the global maximum is 
given by the principal minors of the matrix
$$
\hat \Theta \,\, =\,\, \begin{small}
\begin{bmatrix}
6.5      & 5.5      & 5.97913  & 5.97913  \\
5.5      & 6.5      & 5.97913  & 5.97913  \\
5.97913  & 5.97913  & 6.5      & 5.5 \\
5.97913  & 5.97913  & 5.5      & 6.5
\end{bmatrix}. \end{small}
$$
For the other $305 = 3526-3221$ solutions,  more careful path-tracking
with homotopy methods is needed.
We found $315$ of our critical points to be real. Only
$180$ have real preimages $\hat \Theta$ under the
maximal minor map. Among these,
 $104$ come from positive-definite matrices.
These $104$ are the statistically meaningful critical points.
They include five local maxima.
\end{example}

We conclude this paper by reporting on our computations for $n=5$.
We use the birational parametrization in
Example \ref{ex:reparametrization}.
Our system consists of $15$ rational function equations in $15$ unknowns,
namely the variables $x_{ij}$ of our birational parametrization
and the diagonal entries $\theta_{ii}$ of a 
$5 \times 5$ matrix $\Theta$.
Using 256 threads, we apply \verb+monodromy_solve+ to the $15 \times 15$ system of partial derivatives of $L_u$.
In six days, we already found 29.5 million solutions. 
Hence the ML degree satisfies $\mu_5 \geq 29,500,000$.
Determining $\mu_5$ is a  future project.

In spite of these challenges, we are optimistic that
numerical algebraic geometry will offer some
solutions also for $n \geq 6$. We can run monodromy loops until
a heuristic stopping criterion is satisfied. At that point we
will have gathered a large sample of local maxima
of the likelihood function, complementing those
one finds with local hill-climbing methods.

		\bigskip
		\smallskip

		\noindent {\bf Authors' addresses:}
		
		\smallskip
			
		\noindent Hannah Friedman, UC Berkeley
				\hfill \url{hannahfriedman@berkeley.edu}
		
		\noindent  Bernd Sturmfels, MPI-MiS Leipzig \hfill \url{bernd@mis.mpg.de}

		\noindent  Maksym Zubkov, UC Berkeley \hfill \url{mzubkov@berkeley.edu}

\end{document}